\documentclass[final,1p,times]{elsarticle}

\usepackage{graphicx}
\usepackage{amssymb}
\usepackage{amsthm}
\usepackage{amsmath}
\usepackage{natbib}
\usepackage{comment}
\usepackage{xcolor}

\theoremstyle{definition}
\newtheorem{definition}{Definition}[section]

\newtheorem{remark}[definition]{Remark}
\newtheorem{example}[definition]{Example}

\newtheorem{algorithm}[definition]{Algorithm}
\newtheorem{assumption}[definition]{Assumption}

\theoremstyle{plain}
\newtheorem{lemma}[definition]{Lemma}
\newtheorem{proposition}[definition]{Proposition}
\newtheorem{theorem}[definition]{Theorem}
\newtheorem{corollary}[definition]{Corollary}

\newtheorem{problem}[definition]{Problem}

\def\RR{{\mathbb R}}

\def\I{{\mathcal I}}

\def\B{{\mathcal B}}

\def\C{{\mathbb C}}

\def\N{{\mathbb N}}
\def\Q{{\mathbb Q}}
\def\R{{\mathbb R}}
\def\Ball{{\rm Ball}}

\def\ff{{ f}}

\def\rank{{\rm rank}}

\def\Tr{{\rm Tr}}

\begin{document}

\begin{frontmatter}

\title{Certified Hermite Matrices from Approximate Roots} 
%

\author{Tulay Ayyildiz Akoglu\fnref{Tulayfootnote}}
\address{Karadeniz Technical University, Trabzon, Turkey}
\fntext[Tulayfootnote]{The research of Tulay Ayyildiz Akoglu is partially supported by TUBITAK grant 119F211.}
\ead{tulayaa@ktu.edu.tr}

\author{Agnes Szanto\fnref{Agnesfootnote}} 
\address{North Carolina State University, Raleigh, NC, USA}
\ead{aszanto@ncsu.edu } \fntext[Agnesfootnote]{The research of Agnes Szanto is partially supported by NSF grant  CCF-1813340.} 
%
\begin{abstract}
	
Let $\mathcal{I}=\langle f_1, \ldots, f_m\rangle\subset  \mathbb{Q}[x_1,\dots,x_n]$ be a zero dimensional radical ideal defined by polynomials given with exact rational coefficients. Assume that we are given approximations $\{z_1, \ldots, z_k\}\subset \mathbb{C}^n$ for the common roots $\{\xi_1, \ldots, \xi_k\}=V(\mathcal{I})\subseteq\C^n$. 
In this paper we show how to construct and certify the rational entries of Hermite matrices for $\mathcal{I}$ from the approximate roots $\{z_1, \ldots, z_k\}$. When $\mathcal{I}$ is non-radical, we give methods to construct and certify Hermite matrices for  $\sqrt{\mathcal{I}}$ from approximate roots.  Furthermore, we use signatures of these Hermite matrices to give rational  certificates of non-negativity of a given polynomial over a (possibly positive dimensional) real variety, as well as  certificates  that there is a real root within an $\varepsilon$ distance from a given point $z\in \mathbb{Q}^n$. 
\end{abstract}



\end{frontmatter}
\section{Introduction}

The development of numerical and symbolic techniques to solve systems of polynomial equations resulted in an explosion of applicability, both in terms of the size of the systems efficiently solvable and the reliability of the output.  Nonetheless, many of the results produced by numerical methods are not certified. In this paper, we show how to compute exact Hermite matrices from  approximate roots of polynomials, and how to certify that these Hermite matrices  are correct.

Hermite matrices and Hermite bilinear forms were introduced by Hermite in 1850 \cite{Hermite1850} for univariate polynomials and were extended to the multivariate zero-dimensional setting in \cite{BPR2006, Sti1890}. Hermite matrices  have many applications, including 
counting real roots  \cite{,BPR2006, Hermite1853,Hermite1856} and locating them \cite{AA2016}.  Assume that we are given the  ideal $\mathcal{I}:=\langle f_1, \ldots, f_m\rangle\subset \Q[x_1, \ldots, x_n]$ generated by rational polynomials, and assume that $\dim_\Q \Q[x_1, \ldots, x_n]/\I=k<\infty$. Hermite matrices  have two kinds of definitions (see the precise formulation in Section \ref{subsect:HermiteMatrices} ):
\begin{enumerate}
\item The first definition of Hermite matrices uses the traces of $k^2$ multiplication matrices, each of them of size $k\times k$. The advantage of this definition is that it can be computed exactly, working with rational numbers only. The disadvantage is that it requires the computation of the traces of $k^2$ multiplication matrices. 
\item The second definition uses symmetric functions of the  $k$ common roots of $\I$, counted with multiplicity. 
The advantage of this definition is that it gives a very efficient way to evaluate the entries of the Hermite matrix, assuming that we know the common roots of $\I$ exactly. The disadvantage is that we need to compute the common roots exactly, which may involve working in field extensions of $\Q$. 
\end{enumerate}

In this paper we propose to use the second definition to compute Hermite matrices, but instead of using exact roots, we use approximate roots that can be computed with numerical methods efficiently \cite{HSW2013}.  Once we obtain an approximate Hermite matrix, we use rational number reconstruction  to construct a matrix with rational entries of bounded denominators. Finally, we give a symbolic method which certifies that the rational Hermite matrix we computed is in fact the correct one, corresponding to the exact roots of $\I$.


 The novelty of this work and the difficulty of this problem is  to certify the correctness of the Hermite matrix that we computed with the above heuristic approach. This part of the algorithm is purely symbolic. The main idea is that we use the relationship between multiplication matrices and Hermite matrices to compute a system of multiplication matrices from Hermite matrices and vice versa.   As 
  multiplication matrices, which are also rational matrices,   act as roots of the polynomial system,  we can certify their correctness, which in turn gives us a method to certify Hermite matrices. 
  
  Note that  both of the above definitions of Hermite matrices are continuous in the presence of root multiplicities. However,  our use of approximate roots and multiplication matrices necessitate that we first consider the case when $\I$ is radical. To handle the non-radical case, we use the fact that the maximal non-singular submatrix of the Hermite matrix of $\I$ gives the Hermite matrix of $\sqrt{I}$, so this is what we compute and certify. In both the radical and non-radical cases we were able to give sufficient conditions in terms of the quality of the root approximations and the  size of the rational numbers in the Hermite matrices  that guarantee that our Hermite matrix construction and certification algorithms do not fail.

  Another contribution of this paper is the presentation of 
  two novel applications using the signature of the certified Hermite matrices. The first application is to give a rational certificate that a polynomial  $g\in \Q[x_1, \ldots, x_n]$ is non-negative over a smooth real variety $V(f_1, \ldots, f_s)\cap \R^n$ where $f_1, \ldots, f_s\in \Q[x_1, \ldots, x_n].$ This  application was inspired by \cite{CifuPar2017} where the authors give a method to compute a degree $d$ sum of squares (SOS) decomposition (if one exists) for a non-negative polynomial over a real algebraic set using a finite set of sample points from the (complex) algebraic set.  While the SOS decomposition constructed in \cite{CifuPar2017} is approximate, the construction we give here is exact, using  Hermite matrices with rational entries.

  The second application is to give a rational certificate that for a given point $z\in \Q^n$, $\varepsilon\in \Q_+$
  and $\I:=\langle f_1, \ldots, f_m\rangle\subset  \mathbb{Q}[x_1,\dots,x_n]$  zero dimensional, there exists  $\xi\in V(f_1, \ldots, f_m) \cap\R^n$ such that 
 $$
\|z-\xi\|_2^2\leq \varepsilon.
 $$
  
 A natural question arises about the advantage of this hybrid symbolic--numeric approach over purely symbolic methods to compute Hermite matrices. The answer to this question is not black and white, it depends on the situation. For example, one could use a purely symbolic method computing a   Gr\"obner bases, and then  computing a system of multiplication matrices of the input polynomials and then the traces of certain multiplication matrices  give the entries of Hermite matrices symbolically (just as in the first definition above).  Instead, in this study, we use a symbolic-numeric approach as we assume that we have an efficient and parallelizable numerical method that can compute high precision approximations to all common complex roots for  {\em square subsystems} of $f$, i.e. for $n$ random linear combinations of $f_1, \ldots, f_m$. From these approximate roots we select in a  certified manner  a subset $z_1, \ldots, z_k\in \C^n$ (given with floating point numbers as coordinates) that are approximations solutions of $f\in \Q[x_1, \ldots, x_n]^m$.   The worst case arithmetic complexity  for both the purely symbolic and the hybrid methods are similar (asymptotically bounded by $D^n$ where $D$ is the maximum of the degrees of the input polynomials). In \cite{BFS2004,BFSY2005} the authors give some evidence that Gr\"obner basis techniques for highly overdetermined systems ($m>>n$) can be highly efficient.  The underlying idea behind this is that if one is given many polynomials already, there is only a little more work needed to generate a Gr\"obner basis. On the other hand, for systems that are square or close to square, the limited accuracy and parallelizability of the numerical approach allows to handle larger polynomial systems in practice than with purely symbolic approaches. In \cite{Batesetal2014} they compare the two  different approaches to computing and representing the solutions of polynomial systems: numerical homotopy continuation and symbolic computation.\footnote{We thank Jonathan Hauenstein for pointing out to us the subtleties of using symbolic vs. numeric methods for solving polynomial systems}
 
 This paper is a generalization of our paper \cite{AS2020}, where we considered only the univariate radical case.

  The paper is organized as follows. In the next section we introduce some preliminaries with fundamental definitions such as the Hermite matrices, rational number reconstruction and numerical computation of roots of overdetermined polynomial systems. In Section \ref{sect:ConstructHermite}, we explain how one can construct the exact Hermite matrix using the approximate solutions of the given polynomial system. In Section \ref{sect:CertificateHermite} we give an algorithm to certify that the obtained Hermite matrix is the exact one corresponding to our input polynomials. So far we assume that the ideal $\I$ is radical. In Section \ref{sect:Nonradical} we generalize the Hermite matrix computation and certification algorithm to the case when $\I$ is not radical. In Section \ref{sect:appl1} we present the application of Hermite matrices  to give a rational certificate that a given rational polynomial is non-negative over a real algebraic variety defined by rational polynomials. Finally, in Section \ref{sect:appl2} we give another application that is a rational certificate for the existence of an exact common root of a zero-dimensional ideal  within $\varepsilon$ distance from  a given point in $\Q^n$.

 

  

\section{Preliminaries}

\subsection{Hermite Matrices}\label{subsect:HermiteMatrices}


In this section we give two definitions for the Hermite matrix of a zero-dimensional ideal. The first one defines the matrix  from the common roots of the ideal, and in fact can be used to define Hermite matrices from any set of points. The second definition uses the traces of  multiplication matrices of the factor ring of the ideal, and thus it gives a definition where the entries of the Hermite matrices are rational functions of the coefficients of the polynomial system.

Everything in this section is valid for polynomials over $\R$, so while in the rest of the paper we assume  that our input polynomials  are rational, in this section we present the preliminaries over $\R$. We use the following notation. 	Let $f=(f_1, \ldots, f_m) \in \mathbb{R}[x_1,\dots,x_n]^m$ with $\mathcal{I}=\langle f_1, \ldots, f_m \rangle\subset \RR[x_1,\dots,x_n]$ a  zero-dimensional ideal and  $\mathcal{B}=\{x^{\alpha_1},\dots, x^{\alpha_k} \}$ be a monomial basis for $\mathbb{R}[x_1,\dots,x_n]/\mathcal{I}$.

Let  $\xi_1,\xi_2, \dots, \xi_k\in \C^n$ be the common roots of $\I$, here each root listed as many times as their multiplicity. If $\I$ is radical (which we assume in most of the later sections), each root is distinct. We denote by $z_1, \ldots, z_k\in \C^n$ approximations to the exact roots $\xi_1,\xi_2, \dots, \xi_k$.

The first definition of Hermite matrices is given for any multiset of points in $\C^n$ as  follows:

\begin{definition}\label{def:mult_hermite}
	 Let $g\in \R[x_1, \ldots, x_n]$ and $\mathcal{B}=\{x^{\alpha_1},\dots, x^{\alpha_k} \}$ 
	 be a set of monomials in  $\mathbb{R}[x_1,\dots,x_n]$. Let  
	$z_1,z_2, \dots, z_k\in \C^n$ be points, not necessary distinct.   Then the {\bf Hermite matrix} of $z_1,z_2, \dots, z_k$ with respect to $g$, written in the basis $\B$  is 
	\begin{equation}\label{eqn:Hermite_mtxform}
H_g^\B(z_1,z_2, \dots, z_k):=V^T G V
	\end{equation}
	\noindent
	where 
	$$
	V:=V_\B(z_1,z_2, \dots, z_k)=[z_i^{\alpha_j}]_{i,j=1,\ldots,k}
	$$
	is the Vandermonde matrix of $z_1,z_2, \dots, z_k \in \mathbb{C}^n$ with respect to a monomial set $\mathcal{B}$ and $G$ is a $k\times k$ diagonal matrix with $[G]_{i,i}=g(z_i)$ for $i=1,\dots,k$. We may omit $\B$ from the notation if it is clear from the context.
\end{definition}

\begin{example}When $\B=\{1, x_1, \ldots, x_1^{k-1}\}$ and $g=g(x_1)$ is also univariate, then the Hermite matrices are the same as in the univariate case.  If $z_l=(z_{l,1},\ldots, z_{l,n})\in \C^n$ for $l=1, \ldots, k$ then the entries of the $k\times k$ univariate Hermite matrices are defined by 
$$
[H^\B_g(z_1,\ldots,z_k)] _{i,j}= \sum_{l=1}^{k}g(z_{l,1})z_{l,1}^{i+j-2}.
$$
In particular, for  $g(x)=1$,  if $z_{1,1}, \ldots, z_{k,1}$ all distinct, the entries of the Hermite matrix are the power sum elementary symmetric functions of the first coordinates: 
 	\begin{equation}\label{eq:H1f}
	[H^\B_1(z_1,\ldots,z_k)] _{i,j}= \sum_{l=1}^{k}z_{l,1}^{i+j-2}.
	\end{equation} 
\end{example}

The second definition of Hermite matrices implies that the the entries of the Hermite matrix are rational functions of the coefficients of the defining polynomials of $\I$.
	
	\begin{definition}\label{def:hermite_trace}
	Let $\I\subset\RR[x_1, \ldots, x_n]$ be a zero dimensional ideal  and denote $A:=\RR[x_1, \ldots, x_n]/I$, a finite dimensional vectors space over $\RR$ with $k:=\dim_\RR A$. For any $f\in A$ let  $\mu_f:A\rightarrow A, $ $ p+\I\mapsto p\cdot f +\I$ 
	be the multiplication map by $f$ on $A$. Fix a monomial basis $\B=\{x^{\alpha_1},\dots, x^{\alpha_k} \}$ of $A$, and denote by $M^\B_f$ the $k\times k$ matrix of $\mu_f$ in the basis $\B$. The {\bf Hermite matrix} of $\I$ with respect to $g$, written in the basis $\B$  is 
	$$ H_g^\B(\I)=\left[ \Tr(\mu_{g\cdot x^{\alpha_i+\alpha_j}})\right]_{i,j=1}^{k}=\left[ \Tr(M^\B_{g\cdot x^{\alpha_i+\alpha_j}})\right]_{i,j=1}^{k},
	$$
	where $\Tr$ denotes both the trace of a linear transformation and  the trace of a matrix (note that the trace of a linear transformation is the trace of its matrix in \emph{any} basis).  We may omit $\B$ from the notation when it is clear from the context.
	\end{definition}

The next theorem asserts that the two definitions give the same matrix if we take the exact  common roots of a zero dimensional polynomial system. 

\begin{theorem}[\cite{BPR2006}] Let $\I\subset\RR[x_1, \ldots, x_n]$ be a zero dimensional ideal and $g\in \RR[x_1, \ldots, x_n]$. Let  $k:=\dim \RR[x_1, \ldots, x_n]/ \I$ and $\B=\{x^{\alpha_1}, \ldots, x^{\alpha_k}\}$ be a monomial basis for $\RR[x_1, \ldots, x_n]/ \I $.
Let $\xi_1, \ldots, \xi_k\in \C^n$ be the roots of $\I$, each root listed as many times as their multiplicity.  Then 
$$
H_g^\B(\xi_1, \ldots, \xi_k)=H_g^\B(\I).
$$
\end{theorem}

\begin{remark}
Note that if $f_1, \ldots, f_m ,g \in \Q[x_1, \ldots, x_n]$ then for $\I=\langle f_1, \ldots, f_m\rangle$ we have  $H_g(\I)\in \Q^{k\times k}$, thus $H_g(\xi_1, \ldots, \xi_k)\in \Q^{k\times k}$, even if the roots are not rational.
\end{remark}

Next we present the classical Hermite  theorem  that uses the {\em signature} of Hermite matrices to count real roots of a real polynomial system. First we define the signature of a matrix. 

\begin{definition}\label{def:signature}
	Let $A$ be a real and symmetric matrix. Then the {\bf signature} of A is
	$$\sigma(A):=\# \{{\rm positive ~eigenvalues~of } A\}-\# \{{\rm negative ~eigenvalues ~of } A\}$$
\end{definition}

\noindent
Definition \ref{def:hermite_trace} implies that Hermite Matrices of real polynomial ideals are real and symmetric. The classical univariate Hermite Theorem \cite{Hermite1856}  was generalized to the multivariate case by Pedersen, Roy and Szpirglas  \cite{PRS1993}, and was also proved in \cite{BPR2006} and \cite{CLO2006}:

\begin{theorem}[Multivariate Hermite Theorem]\label{thm:mult_herm}
	Let $\I\subset \RR[x_1,\dots,x_n]$ be  zero dimensional and $\B$ be a monomial basis of  $\RR[x_1,\dots,x_n]/\I$. If 
	$H_g(\I)$ is the Hermite matrix of $~\I$ with respect to $g$ in the basis $\B$, then
$$\sigma(H_g(\I)) = \#\{x \in V_\RR(\I)~|~ g(x)>0\}-\# \{ x \in V_\RR(\I)~|~ g(x)<0\}.
	$$
\end{theorem}

\begin{remark}\label{rmk:signature_compt}
 There are several ways to obtain the signature of a $k \times k$ real and symmetric matrix $M$ without computing the eigenvalues explicitly. We will describe two of these methods here. Note that if $M$ is a rational symmetric matrix, then both of these methods can be computed with exact arithmetic over the rationals. 
 \begin{enumerate}
     \item{\em Using Descartes rule of signs:} 
     Let $p(x)$ be the characteristic polynomial of the given $k\times k$ real and symmetric matrix. Since all eigenvalues of real symmetric matrices are real their characteristic polynomials have only real roots. Then the Descartes Rule of Signs provides that $\sigma(M)$ is the difference between the number of sign variation of the coefficients of $p(x)$ and the number of sign variations of the coefficients of $p(-x)$ (see Proposition 8.4 in \cite{BPR2006}). The complexity of this computation is bounded by $\mathcal{O}(k^4).$
    
     \item{ \em Using $LU$ decomposition:}
     The LU decomposition of real symmetric matrices can be written as 
     $LDL^T$ where $L$ is a special lower triangular matrix with 1's on the diagonal entries and $D$ is a diagonal matrix with the entries $\{ u_{11},\ldots,u_{kk} \}$. These entries are the diagonal entries of the upper triangular matrix $U$ obtained from the LU decomposition.
     By Sylvester Law of inertia, the signature of $M$ is the difference between $\#\{u_{ii}~:~u_{ii}>0 \}$ and $\#\{u_{ii}~:~u_{ii}<0 \}$. The complexity of this computation only comes from the cost of the $LU$ decomposition which can be found via Gaussian Elimination. The complexity of Gaussian elimination is bounded by $\mathcal{O}(k^3).$ 
 \end{enumerate}
 
\end{remark}

 We close this subsection with some definitions that will be used later in this paper. First, in our certification algorithm we need the following property of $\B$ (see \cite{Mourrain1999}):

\begin{definition}\label{def:connect}
Let $\mathcal{B}\subset \mathbb{R}[x_1,\dots,x_n]$. We say that $\B$ (or ${\rm span}_\R(\B)$) is {\em connected to 1} if for all $b\in {\rm span}_\R(\B)$ there exists $b_1, \ldots, b_n\in \B$ such that 
$$b=\sum_{i=1}^n x_ib_i$$
and $\deg(b_i)< \deg(b)$ for $i=1, \ldots, n$.
\end{definition}

 \begin{definition}
  $\mathcal{B}\subset \mathbb{R}[x_1,\dots,x_n]$ be a finite set of monomials. The  {\em extension} of $\mathcal{B}$ is defined by
 	\begin{equation}\label{eqn:B_ext}
 	\mathcal{B}^+ := \mathcal{B} \cup \bigcup_{i=1}^n x_i \mathcal{B}=\left\{b,x_1b,\ldots,x_nb~|~b \in \mathcal{B}\right\}.
 	\end{equation} 
 	
 \end{definition}

\begin{definition}\label{def:ext_Hermite}
	Let $\B$ be a finite set of monomials and assume that $|\B^+|=l$. The {\bf extended Hermite matrix} associated to  points
	$~z_1,\ldots,z_k \in \mathbb{C}^n$ (not necessarily distinct) is
	\begin{equation}\label{eqn:Hermite_mtxform_ext}
	H^+_g:=H^{\B^+}_g(z_1,\ldots, z_k):=(V^+)^T G V^+\in \C^{l\times l}
	\end{equation}
	where $V^+=V_{\mathcal{B}^+}(z_1, \dots, z_k)\in \C^{k\times l}$ and G is the $k\times k$ diagonal matrix with $[G]_{j,j}=g(z_j)$ for $j=1,\dots,k$. 

\end{definition}

\subsection{Rational Number Reconstruction}\label{sect:RNR}

The continued fraction method for a real number $\alpha>0$ can be described as the computation of 
$$
\alpha=\gamma_1=\lfloor\gamma_1\rfloor+\frac{1}{\gamma_2}=\lfloor\gamma_1\rfloor+\frac{1}{\lfloor\gamma_2\rfloor+\frac{1}{\gamma_3}}=\cdots
$$
where $\gamma_1=\alpha$ and $\gamma_{i+1}=\frac{1}{\gamma_i-\lfloor\gamma_i\rfloor}$. We call the rational numbers 
$$
\lfloor\gamma_1\rfloor, \lfloor\gamma_1\rfloor+\frac{1}{\lfloor\gamma_2\rfloor}, \lfloor\gamma_1\rfloor+\frac{1}{\lfloor\gamma_2\rfloor+\frac{1}{\lfloor\gamma_3\rfloor}}, \ldots
$$
the {\em convergents} for $\alpha$. The following theorem (c.f. \cite[Corollary 6.3a]{schrijver1998}) gives bounds on the distance from $\alpha$ that guarantees uniqueness of a rational number with bounded denominator, and shows that if such rational number exists, it is a convergent for $\alpha$.  




\begin{theorem} \cite{schrijver1998}\label{ratrec} 
	There exists a polynomial time algorithm which, for a given rational number $\alpha$ and a natural number $B$ tests if there exists a pair of integers $(p,q)$ with $1\leq q\leq B$ and 
	$$
	\left|\alpha-\frac{p}{q}\right|<\frac{1}{2B^2},
	$$ 
	and if so, finds this unique rational number  $\frac{p}{q}$ as a convergent for $\alpha$. 
\end{theorem}

Approximate solutions are floating point numbers which are obtained from numerical  computations. Using an absolute error bound  $E>0$ on  numerical computations, we can set the denominator bound $B$ such that Theorem \ref{ratrec} provide the unique rational approximation of $\alpha$ as follows:

\begin{corollary}\label{cor:denom_bound}
Given $\alpha\in \R$  and $E>0$ in $\R$ there is at most one rational number with its denominator bounded by $B:=\left\lceil (2E)^{-1/2}\right\rceil$ within the distance $E$ from $\alpha$.

\end{corollary}

\subsection{Numerical certification of non-roots of overdetermined systems}\label{sec:numroots}

In this subsection we summarize some numerical methods from the literature  to compute  approximations of a superset of the common roots of an overdetermined system of polynomials $f=(f_1, \ldots, f_m)\in \Q[x_1, \ldots, x_n]$ where $m>n$ and to give a method to certify that a given point $z$ is \emph{not} an approximate root of $f$. Here we need to assume that $\I=\langle f_1, \ldots, f_m\rangle$ is radical and zero dimensional. The difficulty lies in the fact that numerical methods, such as the homotopy continuation method \cite{HSW2013}, are designed to compute the common roots of square, well-conditioned polynomial systems. Here we describe the approach presented in \cite{Hauenstein-Sottile} to handle overdetermined systems.

  The main idea is that even though the consistency of an overdetermined system is a non-continuous property, the converse, inconsistency of an overdetermined system is a continuous property that can be certified with numerical methods. Similarly, while certifying that a point is approximating an exact root of an overdetermined system with exact rational coefficients cannot be certified with purely numerical methods, certifying that a point is {\em not} an approximate root can be done numerically. This allows us to eliminate in a certified manner roots of a square subsystem of $(f_1, \ldots, f_m)$ that are not roots of all polynomials and to give a certified  upper bound $k$ on the number of common roots of $(f_1, \ldots, f_m)$. With this upper bound we can guarantee that our certification algorithm of Hermite matrices in Section \ref{sect:CertificateHermite} is correct when it gives a certification (although it may also return ``fail''). 

More precisely, \cite{Hauenstein-Sottile} suggests to consider two square systems, each are random linear combinations of the polynomials $f_1, \ldots, f_m$: If $R_1, R_2: \C^m\rightarrow \C^n$ two linear  maps (represented by two random matrices), then for $R_i(f):=R_i\circ f$ we can assume that 
$$
V(f)=V(R_1(f))\cap V(R_2(f)),
$$ which property is satisfied unless $R_1$ and  $R_2$ are from  Zariski closed subsets of all linear transformations. Similarly, for $i=1,2$,  we can assume that the roots of $V(R_i(f))$ are finite and all distinct,  this property is also satisfied unless $R_i$ is from a Zariski closed  subset of all linear transformations (c.f. \cite[Section 3]{Hauenstein-Sottile}).
Using numerical homotopy continuation methods and $\alpha$-theory, one can compute and certify all approximate roots of both $R_1(f)$ and $R_2(f)$, see details on this part of the algorithm in  \cite[Section 2]{Hauenstein-Sottile}. 

The following idea is a slight modification of \cite[Section 3]{Hauenstein-Sottile} allowing to discard approximate roots of $R_1(f)$ and $R_2(f)$ that do not approximate roots in $V(f)$. First note that for any approximate root $z$ approximating an exact root $\xi\in V(R_i(f))$ for $i=1,2$,  we  can give upper bounds for  $\|{z}-{\xi}\|$ using two times the $\beta$ function defined in $\alpha$-theory (c.f. \cite[Ch 8, Theorem 2]{Blumetal1998}), even without knowing the exact root. Thus, for $i=1,2$ denote by 
$\tilde{V}(R_i(f))$ the set of pairs $(z,\varepsilon)\in \C^n\times \R_+$ where $z$ is one of the approximate roots computed for $R_i(f)$ and $\varepsilon$ is an upper bound of the distance of $z$ from the exact root it approximates. Fix $(z,\varepsilon)\in \tilde{V}(R_1(f))$ and define
$$
S_{(z,\varepsilon)}:=\left\{(z',\varepsilon')\in \tilde{V}(R_2(f))\;:\: \|z-z'\|\leq \varepsilon+\varepsilon'\right\}.
$$
If $|S_{(z,\varepsilon)}|>1$ then we need to refine $z$ and all $z'$ such that $(z',\varepsilon')\in S_{(z,\varepsilon)}$   using  Newton's method w.r.t $R_1(f)$ and $R_2(f)$ respectively, until one gets $S_{(z,\varepsilon)}=\emptyset$ or $|S_{(z,\varepsilon)}|=1$.  If $S_z=\emptyset$ then we can discard $(z,\varepsilon)$ since it cannot approximate an exact root in $V(f)=V(R_1(f))\cap V(R_2(f))$. If $S_z$ has one element $(z', \varepsilon')$, suppose $z$ approximates an exact root $\xi\in V(R_1(f))$, $z'$ approximates an exact root $\xi'\in V(R_2(f))$ but $\xi\neq \xi'$. Then we can compute refinements $z_k$ and $z'_k$ using $k$ iterations of Newton's method starting from $z$ and $z'$  w.r.t $R_1(f)$ and $R_2(f)$ respectively, such that 
\begin{equation}\label{eq:zdist}
 \|{z}_k-{ z'}_k\|> \varepsilon_k +\varepsilon'_k. 
\end{equation}
 where $\varepsilon_k\leq \frac{1}{2^{2^k-1}}\varepsilon$ is a bound on $\|z_k-\xi\|$, and $\varepsilon'_k$ is a bound for $\|z'_k-\xi'\|$. If we find a $k$ such that the inequality (\ref{eq:zdist}) is satisfied then we discard $(z,\varepsilon)$, otherwise we keep it. We repeat the above procedure for all elements in $\tilde{V}(R_1(f))\cup\tilde{V}(R_2(f))$.

Note that the above method never eliminates points  that were approximating roots in $V(f)$, but may leave in points that were not near $V(f)$. Thus, as a consequence, we can always guarantee that the input $z_1, \ldots, z_k$ for Algorithm \ref{alg:apprHerm} below, to compute an approximate Hermite matrix, is a {\em superset} of an approximation of $V(f)$, and in particular the above method gives a certification  that $k\geq \dim \Q[x_1, \ldots, x_n]/\I$. On the other hand, our main symbolic-numeric certification Algorithm  \ref{alg:cert} for Hermite matrices   will always fail when there are superfluous points among the input. Thus, using the assumption $k\geq \dim \Q[x_1, \ldots, x_n]/\I$,  the Hermite matrices that we certify successfully will correspond to all roots of $V(f)$ (see Theorem \ref{thm:maincert}).

In the rest of the paper we assume that we already computed a set $\{z_1, \ldots, z_k\}\subset \C^n$ that contain an approximate root for each root  in $V(f)$, i.e. if $V(f)=\{\xi_1, \ldots, \xi_{k'}\}$ then $k'\leq k$ and for $i=1, \ldots ,k'$ there exists $j_i\in \{1, \ldots, k\}$ such that the Newton iteration starting from $z_{j_i}$ quadratically converges to $\xi_i$. Moreover, using the $\beta$ function from $\alpha$-theory as above,  we assume that we have a certified bound $E\in \R_{+}$ that we call {\em accuracy}, such that
\begin{equation}\label{eq:E}
\left\|\xi_i-z_{j_i}\right\|_2\leq E \quad i=1, \ldots, k'.
\end{equation}


\section{Constructing Rational Hermite Matrices}\label{sect:ConstructHermite}


In this section we construct a rational matrix $H_1^+\in \Q^{l\times l}$ from points $z_1, \ldots, z_k$ given with limited precision, using the definition of Hermite matrices  in Defintion \ref{def:ext_Hermite} and rational number reconstruction. 

Let $z_1, \ldots, z_k\in \C^n$ with $z_{i}=(z_{i,1},\ldots,z_{i,n})$ for $i=1,\ldots,k$. 
Let $\B=\{x^{\alpha_1}, \ldots, x^{\alpha_k}\} $ be basis for 
$\RR[x_1,\ldots,x_n]/\I(z_1,\ldots,z_k)$, and we use $\B^+$ as described in (\ref{eqn:B_ext}). In the certification algorithm below we will assume that $\B$ is connected to 1 as in Definition \ref{def:connect}, but the algorithm of this section works for arbitrary $\B$. 
 Algorithm \ref{alg:apprHerm} below computes the  matrix $H_1^+$ from the Hermite matrix $H_1^{\B^+}(z_1,z_2,\ldots,z_k)$ with respect to $\B^+$ by applying rational number reconstruction.

As part of the input of Algorithm \ref{alg:apprHerm}, we also use  quantities  $E, M\in \RR_+$, where $E$  is  an  upper bound for the accuracy of each $z_i$ for $i=1, \ldots, k$ and $M$ is an upper bound for the absolute values of the coordinates of the exact common roots of $\I$. We assume that $E$ is computed as part of the numerical method computing $z_1, \ldots, z_k$, as described in Section \ref{sec:numroots}. In this section we use  $E$ and $M$ to estimate the denominator in the rational number reconstruction for each entry of $H_1^+$ using  Proposition \ref{prop:denom_bound} below.\\

\begin{algorithm}[Hermite Matrix Computation]\label{alg:apprHerm}
$\;$
\begin{description}
	\item[Input:] 
	 $\B=\{x^{\alpha_1}, \ldots, x^{\alpha_k} \}$ and $\B^+$ as in (\ref{eqn:B_ext}) with $|\B^+|=l$  for $k,l\in \N$.\\
	 $E, M\in\R_+$ and $z_1, \ldots, z_k\in \C^n$ such that $\|z_i\|_\infty \leq M-E$ for $i=1, \ldots, k$ and $E$ is as in (\ref{eq:E}). 
	 
	 \item[Output:] $H_1^+\in \Q^{l\times l}$ with rows and columns indexed by the elements of $\B^+$. 
\noindent
\item[~~~~~~1:] Compute the extended Hermite matrix $H_1^{\B^+}(z_1,z_2,\ldots,z_k)$ using Definition \ref{def:ext_Hermite} with respect to the auxiliary function $g=1$ and the monomials in $\mathcal{B}^+$.  

\item[~~~~~~2:] Rationalize each entry of  the approximate Hermite matrix $H_1^{\B^+}(z_1,z_2,\ldots,z_k)$ using rational number reconstruction as explained in Subsection \ref{sect:RNR}. 
For the $(i,j)$-th entry of the $H_1^{\B^+}(z_1,z_2,\ldots,z_k)$, we use the following denominator bound:
	\begin{equation}\label{eqn:denom_bound_mtx}
	B_{ij}:=\left\lceil (2Eknd_{i,j} M^{d_{i,j}-1})^{-1/2}\right\rceil,
	\end{equation}
where $d_{i,j}=\deg b_i+\deg b_j$ and $b_i$ and $ b_j$ are the $i$-th and $j$-th elements of $\B^+$ respectively, for $1\leq i,j\leq l$. (See Proposition \ref{prop:denom_bound} below for obtaining this bound.) Return the resulting rational matrix.
\end{description}
\end{algorithm}


We need the following proposition to get the bounds in (\ref{eqn:denom_bound_mtx}) for the denominators of the entries of the Hermite matrix. 

\begin{proposition}\label{prop:denom_bound}
Given $\I=\langle f_1,\ldots,f_m\rangle  \subset \Q[x_1, \ldots, x_n]$ zero dimensional radical ideal with $V(\I)=\{\xi_1, \ldots, \xi_k\}\subset\C^n.$ Suppose $E>0$ and $z_1, \ldots, z_k\in \C^n$ are such that 
$$
\|\xi_i-z_i\|_2< E.
$$ 
Let $M>0$ such that for all $i=1, \ldots, k$ and $y\in \Ball(z_i, E)=\{x\in \C^n: \|x-z_i\|_2<E\}$
$$
\|y\|_\infty \leq M .
$$
Let $\alpha\in \N^n$ with $d:=|\alpha|$. Then  we have   \begin{eqnarray}\label{eqn:errorbound}\left|\sum_{i=1}^k \xi_i^\alpha - \sum_{i=1}^k z_i^\alpha\right| \leq E k n d M^{d-1} .
\end{eqnarray}
Furthermore, there is at most one rational number within 
$E k n d M^{d-1}$ distance from $\sum_{i=1}^k z_i^\alpha$ with denominator bounded by 
$$    B:=\left\lceil (2Ek n d M^{d-1})^{-1/2}\right\rceil .   $$ 
\end{proposition}

\begin{proof}
Fix $\alpha\in \N^n$ with $d:=|\alpha|$. First note that $\sum_{i=1}^k z_i^\alpha$ is a polynomial in the $nk$ coordinates of $z_1, \ldots, z_k$ of degree $d$. Using a multivatiate version of Taylor's Theorem \cite{Apostol1969}, there exist  $R_{i,j}(z_1, \ldots, z_k)$ for $i=1, \ldots, k$ and $j=1, \ldots, n$  such that 
$$
\sum_{i=1}^k z_i^\alpha-\sum_{i=1}^k \xi_i^\alpha=\sum_{i=1}^k\sum_{j=1}^n R_{i,j}(z_1, \ldots, z_k)(z_{i,j}-\xi_{i,j}).
$$
Moreover, 
$$
|R_{i,j}(z_1, \ldots, z_k)|\leq \max_{s,t}\max_{y_s\in {\Ball}(z_s, E) }\left|\frac{\partial z_s^\alpha}{\partial z_{s,t}}(y_s)\right|\leq dM^{d-1}.
$$
Thus we get 
$$
\left|\sum_{i=1}^k z_i^\alpha-\sum_{i=1}^k \xi_i^\alpha\right|\leq \sum_{j=1}^n\sum_{i=1}^kdM^{d-1}\left| z_{i,j}-\xi_{i,j}\right|\leq kndM^{d-1}E.
$$
The second claim is  straightforward from Corollary \ref{cor:denom_bound} using $E'=E k n d M^{d-1}$. 
\end{proof}

The next theorem gives sufficient conditions for Algorithm \ref{alg:apprHerm}  to correctly  compute the exact Hermite matrices $H_1^+$ for $\I$ from the approximate points $z_1, \ldots, z_k$.

\begin{theorem} \label{thm:correct1} Let $\I=\langle f_1, \ldots, f_m\rangle  \subset \Q[x_1,\dots,x_n]$ be a zero dimensional radical ideal with $\B=\{x^{\alpha_1}, \ldots, x^{\alpha_k}\}$ a basis for $\Q[x_1,\dots,x_n]/\I.$ Denote $V_\C(\I)=\{\xi_1, \ldots, \xi_k\}\subset \C^n$. Let $E>0$ and $z_1, \ldots, z_{k}\in \C^n$ such that for each $i\in \{1, \ldots , k\}$ there exists a unique $j_i$ such that 
$$\|z_{j_i}-\xi_{i}\|_2<E.
$$
Assume further that  $\|z_{i}\|_\infty\leq M-E$ for all $i=1, \ldots, k$. Finally, assume that for  $x^\alpha=x_ix_jx^{\alpha_t} \,x^{\alpha_s}$  for $i,j=1, \ldots n$ and $t,s=1, \ldots, k$, the denominator of $\sum_{j=1}^k \xi_i^\alpha\in \Q$ is at most $\lceil (2E k n |\alpha| M^{|\alpha|-1})^{-1/2}\rceil$. Then Algorithm \ref{alg:apprHerm}  computes the exact Hermite matrices $H_1^{\B^+}(\I)$. 
\end{theorem}

\begin{proof}
 Since   $\|z_i-\xi_{j_i}\|_2<E$ and $\|z_{i}\|_\infty\leq M-E$, we have that $\|\xi_{i}\|_\infty\leq M$, thus the assumptions  of Proposition \ref{prop:denom_bound} are satisfied. Therefore,  for all $\alpha\in \N^n$ as in the claim, we have that there is  at most one rational number within $E k n |\alpha| M^{|\alpha|-1}$ distance from $\sum_{i=1}^k z_i^\alpha$ with denominators bounded by $\lceil (2E k |\alpha| M^{|\alpha|-1})^{-1/2}\rceil$. By Proposition \ref{prop:denom_bound} we can see that  $\sum_{i=1}^k\xi_i^\alpha$ is within that distance from $\sum_{i=1}^k z_i^\alpha$, and   using our assumption on the denominator of 
 $\sum_{i=1}^k\xi_i^\alpha\in \Q$, by Theorem  \ref{ratrec} the rational number reconstruction algorithm finds $\sum_{i=1}^k\xi_i^\alpha$. Thus, the entries of $H_1^+$ computed by Algorithm \ref{alg:apprHerm} are the same as the entries of $H_1^{\B^+}(\I)$ as claimed.
\end{proof}

\begin{remark}
The assumption of Theorem \ref{thm:correct1} that the denominator of $\sum_{j=1}^k \xi_i^\alpha\in \Q$ is at most $\lceil (2E k n  |\alpha| M^{|\alpha|-1})^{-1/2}\rceil$ can be achieved by improving the accuracy $E$ of the approximate roots $z_1, \ldots, z_k$. If we assume that  $z_1, \ldots, z_k$  are all approximate roots for a square subsystem of $f$, we can use Newton iterations to quadratically converge to the exact roots, thus decrease $E$. Meanwhile, the other quantities in this bound ($k,n,|\alpha|, M$) are fixed, so with enough iterations we  increase $\lceil (2E k n  |\alpha| M^{|\alpha|-1})^{-1/2}\rceil$ to satisfy the condition of the Theorem. One could in theory study a priori bounds on how small $E$ has to be to satisfy this condition,  (we did a similar analysis in \cite{AHS2018}), but in this paper we let our Hermite matrix certification algorithm reject cases when the accuracy of the root approximation is not sufficiently good.   \end{remark}

\section{Certification of the Exact Hermite Matrix}\label{sect:CertificateHermite}

Let $f=(f_1, \ldots, f_m) \in \mathbb{Q}[x_1,\dots,x_n]^m$ be a rational polynomial system  with zero dimensional radical ideal $\mathcal{I}=\langle f_1, \ldots, f_m\rangle$. In the previous section we computed a matrix $H_1^+$ with rows and columns corresponding to $\B^+$, where $\B=\{x^{\alpha_1}, \ldots, x^{\alpha_k}\}$. In this section we certify if this matrix is the the extended Hermite matrix of $\I$ and we also compute $H_g(\I)$ for any polynomial $g\in \Q[x_1, \ldots, x_n]$, as long as $k\geq \dim\Q[x_1, \ldots, x_n]/\I$. Here we assume that $\B$ is connected to 1 as in Definition \ref{def:connect}. The following algorithm is purely symbolic:

\begin{algorithm}[Hermite Matrix Certification]\label{alg:cert}
$\;$
\begin{description}
	\item[Input:]
	 $\ff=(f_1, \ldots, f_m) \in \Q[x_1,\dots,x_n]^m$ with $\I=\langle f_1, \ldots, f_m\rangle$ zero dimensional and radical;\\
	 $g \in \Q[x_1,\ldots,x_n]$,\\
	 $\B=\{x^{\alpha_1}, \ldots, x^{\alpha_k}\}$ connected to 1  with $|\B^+|=l$  for some $k,l\in \N$\\
	 $H_1^+ \in \Q^{l \times l}$   with rows and columns indexed by the elements of $\B^+$. 
	
	\item[Output:] 
	The certified $H_1(\I)$ and $H_g(\I)$, or Fail.
	
	\item[~~~~~~1:] $H_1 \leftarrow k\times k$ submatrix of $H_1^+$ with rows and columns corresponding to $\B$.  \\
	\noindent $H_1^{x_s} \leftarrow$ $k\times k$ submatrix of $H_1^+$ with rows corresponding to $\B$ and columns corresponding to $x_s\B$ for $s=1, \ldots, n$.
	
	\item[~~~~~~2:]{\bf If}  
	$\rank~ H_1 = \rank~ H_1^+=k,$
	{\bf then} $M_s \leftarrow H_1^{-1}\cdot H_1^{x_s}$ for $s=1, \ldots, n$. {\bf else} return Fail.
	
	\item[~~~~~~3:] {\bf For} $s=1, \ldots,n$,  $i,j=1, \ldots, k$ \\
	 \hspace {2cm} {\bf if}  $x_sx^{\alpha_i}=x^{\alpha_j}$ {\bf and} $\left[M_s\right]_{i,*}\neq {\bf e}_j^T$ {\bf then return} Fail.
	 
	\item[~~~~~~4:]
	Let $c_1, \ldots, c_n$ be either new parameters or generic elements of $\Q$.\\
	$p(\lambda) \leftarrow$ characteristic polynomial polynomial to $\sum_{i=1}^nc_iM_i$.\\
	  {\bf  if} $\gcd(p(\lambda), p'(\lambda))\neq 1$ {\bf return} Fail.\\

	\item[~~~~~~5:] {\bf If} 
	$$ M_i\cdot M_j=M_j\cdot M_i \quad 1\leq i<j\leq n
	$$ and
	$$
	f_i(M_1,M_2,\ldots,M_n)=0 \hbox{ for } i=1,\ldots,m,
	$$
	 {\bf then} we certified that $M_i$ is the transpose of the multiplication matrix of $\I$ with respect to $x_i$ in the basis $\B$ for all $i=1,\ldots,n$. \\
	 {\bf Else return} Fail.
		
	\item[~~~~~~6:] {\bf For} $i,j=1, \ldots, l$  {\bf if} 
	$${\rm Tr}((b_i\cdot b_j)(M_1,M_2,\ldots,M_n))\neq \left[H_1\right]_{i,j}
	$$ 
	where  $b_i$ and $b_j$ are the $i$-th and $j$-th elements of $\B$ respectively, and  $(b_i\cdot b_j)(M_1,M_2,\ldots,M_n)$ is the matrix obtained by evalutaing the polynomial $b_i\cdot b_j$ in the matrices $M_1,M_2,\ldots,M_n$\\
	{\bf then return } Fail.\\
	{\bf Else} we certified  $H_1=H_1(\I)$.  
	
	\item[~~~~~~7:]  {\bf Return} $H_1$ and 
	$H_g \leftarrow H_1\cdot g(M_1,\ldots,M_n).$
	
	\end{description}
\end{algorithm}

We have the following result on the correctness of Algorithm \ref{alg:cert}.

\begin{theorem}\label{thm:maincert}
Let $f=(f_1, \ldots, f_m)\in \Q[x_1, \ldots, x_n]$ with $\I=\langle f_1, \ldots, f_m\rangle  \subset \Q[x_1,\dots,x_n]$,   
	 $g \in \Q[x_1,\ldots,x_n]$,
	 $\B=\{x^{\alpha_1}, \ldots, x^{\alpha_k}\}$ connected to 1  with $|\B^+|=l$  for some $k,l\in \N$. Let 
	 $H_1^+ \in \Q^{l \times l}$ be a matrix  with rows and columns indexed by the elements of $\B^+$. If $k\geq \dim\Q[x_1, \ldots, x_n]/\I$ and Algorithm \ref{alg:cert} does not return  Fail then $\I$ is radical, $\B$ is a basis for $\Q[x_1, \ldots, x_n]/\I$ and the output satisfies 
	 $$H_1=H_1^\B(\I) \; \text{ and } \;  H_g=H_g^\B(\I).$$
\end{theorem}

\begin{proof} Assume Algorithm \ref{alg:cert} did not fail, and let $M_1, \ldots, M_n$ be the matrices computed in Step 2. Define the set of polynomials ${\mathcal F}$  from the columns of $M_1, \ldots, M_n$ as follows 
	$$
	{\mathcal F} :=\left\{ x_jx^{\alpha_t}-\sum_{i=1}^{k}[M_j]_{i,t}\,x^{\alpha_i}\;:\; j=1, \ldots, n, t=1, \ldots, k\right\}.
	$$
	Let $J:=\langle {\mathcal F} \rangle \subset \Q[x_1, \ldots, x_n]$. In \cite[Theorem 3.1]{Mourrain1999} it is proved that   $\B$ forms a basis for $\Q[x_1, \ldots, x_n]/J$ and ${\mathcal F}$ is a border basis for $J$ if the following three conditions are satisfied: 1. the map $N:{\rm span}(\B^+)\rightarrow {\rm span}(\B)$  defined as $N(x_jx^{\alpha_t}):=\sum_{i=1}^{k}[M_j]_{i,t}x^{\alpha_i}$ satisfy $N\left|_\B\right. = {\rm Id};\;$ 2. the space spanned by $\B$ is connected to $1;\;$ 3. $\{M_1, \ldots, M_n\}$ is a set of pairwise commuting matrices. The first condition is certified in Step 3, the second is an assumption, and the third is certified in Step 5. Also by   Step 5 we have  
$$
	f_i(M_1,M_2,\ldots,M_n)=0 \hbox{ for } i=1,\ldots,m, 
	$$
	 thus $\I\subset J$. Since $\dim \Q[x_1, \ldots, x_n]/J=k $ and $\dim \Q[x_1, \ldots, x_n]/\I\leq k $ by assumption, we must have that $\I=J$ and for $s=1, \ldots, n$, $M_s$ is the transpose of the matrix of the multiplication map 
	 $$\mu_{x_s}: \Q[x_1, \ldots, x_n]/\I\rightarrow \Q[x_1, \ldots, x_n]/\I, \quad [p]\mapsto [x_sp],
	 $$ with respect to the basis $\B$.   Step 6 certifies that the entries of $H_1^+$ are correct using Definition \ref{def:hermite_trace}. Let $p$ be the characteristic polynomial as in  Step 4. Since $\gcd(p(\lambda), p'(\lambda))\neq 1$ we have that  $p$ has $k$ distinct roots, so $M_1, \ldots, M_n$ are simultaneously diagonalizable, and we have for $g\in \Q[x_1, \ldots, x_n]$
	 $$
	 g(M_1, \ldots, M_n)= V^{-1}GV
	 $$
	 where $V=V_\B(\xi_1, \ldots, \xi_k)$ is the Vandermonde matrix of the exact roots of $V(\I)=\{\xi_1, \ldots, \xi_k\}$ with respect to $\B$ and $G$ is the diagonal matrix ${\rm diag}(g(\xi_1), \ldots, g(\xi_k))$. This gives  
	 $$
	 H_1\cdot g(M_1, \ldots, M_n)=(V^TV)\cdot (V^{-1}GV)=V^TGV=H_g.
	 $$
	  Thus,  once $H_1$ and $M_1, \ldots, M_n$ are certified, we have computed the certified matrix $H_g(\I)$. 
\end{proof}

\begin{remark} 
 For $\B=\{1, x_1,  \ldots, x_1^{k-1}\}$ we can guarantee that 
 $$\rank (H_1)=\rank (H_1^+)=k$$
 by checking that the difference between the first coordinates of $z_i$ and $z_j$ is at least  $2E$ for all $1\leq i<j\leq n$. In this case, since we assume that   $\|z_i-\xi_{j_i}\|_2<E$, then we must have that the first coordinates of $\xi_1, \ldots, \xi_k$ are all distinct. Since $H_1(\I)=V^TV$ where $V$ is the usual univariate Vandermonde matrix of these first coordinates of $\xi_1, \ldots, \xi_k$, we get that $H_1(\I)$ is invertible. Since $H_1^+(\I)=V_{\B^+}^TV_{\B^+}$, and $\rank(V_{\B^+})=k$, we have that  $\rank (H_1)=\rank (H_1^+)=k.$ For general $\B$ there are no such simple conditions, but the SVD of the Vandermonde matrix of $z_1, \ldots, z_k$ with respect to $\B$ gives a good indication that $\B$ is a basis for both $\C[x_1, \ldots, x_n]/I(z_1, \ldots, z_k)$ and $\C[x_1, \ldots, x_n]/I(\xi_1, \ldots, \xi_k)$. 
Also for $\B=\{1, x_1, \ldots,x_1^{k-1}\} $, Step 3 is simply checking if  $M_1$ has a companion matrix structure. In this case, in Step 6 we can certify $H_1^+$ using the Newton-Girard formulae (c.f. \cite{Macdonald79}),  computing the power sum elementary symmetric functions for degrees $1,\ldots, 2k$   from the coefficients of $p(x_1)$, and constructing the matrix $H_1^+$ according to (\ref{eq:H1f}). 
\end{remark}



\section{Extension to the non-radical case}\label{sect:Nonradical}

So far in this paper we assumed that the ideal $\I$ is radical and zero-dimensional. In this section, we describe what we can certify when we drop the assumption of radicality. 

First note that  Definition \ref{def:mult_hermite} for the Hermite matrix in terms of roots is well defined and continuous even if some of the roots are repeated. This is no longer true when we try to define the multiplication matrices from the roots. Using a notation similar to the previous section, for  $\B$ a basis for $\RR[x_1, \ldots, x_n]/\I$ and $g\in \Q[x_1, \ldots, x_n]$, denote by $M^\B_{g}$ the transpose of the matrix of multiplication $\mu_g:\Q[x_1, \ldots, x_n]/\I\rightarrow \Q[x_1, \ldots, x_n]/\I, \; [p]\mapsto [gp]$ in the basis $\B$. While in the case when $\I$ is radical we have 
$$M^\B_{g}=V_\B\, G\, V_\B^{-1},
$$
 with  $V_\B=[b(\xi)]_{\xi\in V(\I), b\in \B}$ and $G={\rm diag}(g(\xi)\,:\,\xi\in V(\I))$,  $M^\B_{g}$ is not diagonalizable when $\I$ is not radical. In particular,  $V_\B(\I)$ and $H_1^\B(I)$ are not invertible, so we cannot use them to compute multiplication matrices in the certification algorithm as above.

To overcome this difficulty, we notice that while in the non-radical case $H_1^\B(\I)$ is not invertible,  its maximal non-singular submatrix  is  the Hermite matrix  $H_1^{\bar{\B}}(\sqrt{\I})$ of the radical of $\I$, with respect to a subset $\bar{\B}\subset \B$ that is a basis for  $\RR[x_1, \ldots, x_n]/\sqrt{\I}$ (c.f. \cite{JanRonSza2007}). Denote by $\bar{H}_1$ this maximal non-singular submatrix of $H_1^\B(\I)$, with rows and columns corresponding to $\bar{\B}\subset \B$  and by $\bar{H}_1^{x_i}$ the submatrix of $H_1^{\B+}(\I)$ corresponding to rows indexed by $\bar{\B}$ and columns indexed by $x_i\cdot\bar{\B}$. Then we get the transpose of the matrix of multiplication by $x_i$ in  $\Q[x_1, \ldots, x_n]/\sqrt{\I}$ w.r.t. the basis $\bar{\B}$ by
$$
\bar{M}_{i}:= M_{x_i}^{\bar{\B}} = \bar{H}_1^{-1}\cdot \bar{H}_1^{x_i}\quad i=1, \ldots, n.
$$
Thus we can use $\bar{H}_1^{\bar{\B}^+}$ to compute the multiplication matrices of the radical $\sqrt{\I}$ and certify them.

We modify Algorithm \ref{alg:apprHerm} to output the maximal non-singular submatrix $\bar{H}_1$ of $H_1$ with rows and columns corresponding to $\bar{\B}\subset \B$, and its extension $\bar{H}^+_1$ to the basis $\bar{\B}^+$.

\begin{algorithm}[Hermite Matrix Computation - Non Radical Case]\label{alg:apprHerm-nonrad}
$\;$
\begin{description}
	\item[Input:] 
	  $k\in \N$, $\B=\{ x^{\alpha_1}, \ldots, x^{\alpha_k}\}$ connected to 1,
	 $E, M\in\R_+$ and $z_1, \ldots, z_k\in \C^n$ such that $E$ is a bound on the accuracy of $z_i$ and  $\|z_i\|_\infty \leq M-E$ for $i=1, \ldots, k$. 
	 \item[Output:] Fail or $\bar{k}\leq k\in \N$, $\bar{\B} \subset \B$ connected to 1, and $\bar{H}_1^+\in \Q^{\bar{l}\times \bar{l}}$ with rows and columns indexed by the elements of $\bar{\B}^+$
	 such that 
	 $$\rank~\bar{H}_1 = \rank~ \bar{H}_1^+=\bar{k},
	 $$
	 where $\bar{H}_1$ is the submatrix of $\bar{H}_1^+$ with rows and columns corresponding to $\bar{\B}$.

\item[~~~~~~1:] $H_1^{\B^+}(z_1,z_2,\ldots,z_k)\leftarrow \text{\sc Hermite Matrix Computation}\left(\B, \B^+, E, M, \{z_1, \ldots, z_k\}\right)$\\ (see Algorithm \ref{alg:apprHerm}).
Denote the resulting matrix by $H_1^+$ and its  submatrix corresponding to rows and columns of $\B$ by $H_1$.
\item[~~~~~~2:] Find $\bar{k}$ maximal and   $\bar{\B}\subset \B$ connected to 1 with $|\bar{\B}|=\bar{k}$ such that  the submatrix $\bar{H}_1$ of $H_1$, corresponding to rows and and columns in $\bar{\B}$, have the same rank as $H_1$. 

\item[~~~~~~3:] If $\rank (H_1^+)>\bar{k}$ return Fail. Else return  $\bar{H}_1^+$, the submatrix of $H_1^+$ with rows and columns corresponding to $\bar{\B}^+$.  
\end{description}
\end{algorithm}

\begin{theorem} \label{thm:correct2} Let $\B=\{ x^{\alpha_1}, \ldots, x^{\alpha_k}\}$, $\I=\langle f_1, \ldots, f_m\rangle  \subset \Q[x_1,\dots,x_n]$ be a zero dimensional  ideal  $V_\C(\I)=\{\xi_1, \ldots, \xi_{\bar{k}}\}\subset \C^n$ for some $\bar{k}\in \N$ 
with the multiplicity of $\xi_i$ denoted by $k_i$, where $k:=\sum_{i=1}^{\bar{k}} k_i=\dim \Q[x_1,\dots,x_n]/\I$. 
Let $E>0$ and $z_1, \ldots, z_{k}\in \C^n$, not necessarily all distinct. Assume that for each $i\in \{1, \ldots , \bar{k}\}$ there exists a multi-subset $Z_i$ of the multiset $\{z_1, \ldots, z_{k}\}$ such that   for $z\in Z_i$
$$\|z-\xi_{i}\|_2<E.
$$
Furthermore, assume that the multiset $Z_i$ has $k_i$ elements, counted with multiplicity,   $Z_i\cap Z_j=\emptyset$ for $1\leq i< j\leq \bar{k}$ and $k=\sum_{i=1}^{\bar{k}}k_i$. Assume further that  $\|z_{i}\|_\infty\leq M-E$ for all $i=1, \ldots, k$. Finally, assume that for  $x^\alpha=x_ix_jx^{\alpha_t} \,x^{\alpha_s}$  for $i,j=1, \ldots n$ and $t,s=1, \ldots, k$, the denominator of $\sum_{j=1}^{\bar{k}}k_i \xi_i^\alpha\in \Q$ is at most $\lceil (2E k n |\alpha| M^{|\alpha|-1})^{-1/2}\rceil$. Then Step 1 computes the exact Hermite matrices $H_1^+=H_1^{\B^+}(\I)$. Moreover, if $\bar{\B}$  defined in Step 2 forms  a basis of $ \Q[x_1,\dots,x_n]/\sqrt{\I}$ then $\rank (H_1^+)=\rank(\bar{H}_1)=\bar{k}$ and $\bar{H}_1^+=H_1^{\bar{\B}^+}(\sqrt{\I})$.
\end{theorem}

\begin{proof}
 Since   $\|z_i-\xi_{j_i}\|_2<E$ and $\|z_{i}\|_\infty\leq M-E$, we have that $\|\xi_{i}\|_\infty\leq M$, thus the assumptions  of Proposition \ref{prop:denom_bound} are satisfied. Therefore,  for all $\alpha\in \N^n$ as in the claim, we have that there is  at most one rational number within $E k n |\alpha| M^{|\alpha|-1}$ distance from $\sum_{i=1}^k z_i^\alpha$ with denominators bounded by $\lceil (2E k |\alpha| M^{|\alpha|-1})^{-1/2}\rceil$. We can take a limit argument from $k$ distinct points  to the multiset $\{(\xi_1, k_1), \ldots, (\xi_{\bar{k}}, k_{\bar{k}})\}$ and apply Proposition \ref{prop:denom_bound} to this multiset and get  that  $\sum_{i=1}^{\bar{k}}k_i\xi_i^\alpha$ is within that distance from $\sum_{i=1}^k z_i^\alpha$.  Using our assumption on the denominator of 
 $\sum_{i=1}^k\xi_i^\alpha\in \Q$, by Theorem  \ref{ratrec} the rational number reconstruction algorithm finds $\sum_{i=1}^{\bar{k}}k_i\xi_i^\alpha$. Thus, the entries of $H_1^+$ computed by Algorithm \ref{alg:apprHerm} are the same as the entries of $H_1^+(\I)$ as claimed.\\
 To prove the last claim, if  $\bar{\B}$ forms a basis for $ \Q[x_1,\dots,x_n]/\sqrt{\I}$ then  $V_{\bar{\B}}=[b(\xi_i)]_{i=1,\ldots, \bar{k}, \, b\in \bar{\B}}$ is a square submatrix of maximal rank of $V_{\B^+}(I)=[b(\xi)]_{\xi\in V(\I), \, b\in \B^+}$ with each roots in $V(\I)$ listed as many times as their multiplicity. Since $H_1^+=H_1^{\B+}(\I)=V_{\B^+}(I)^TV_{\B^+}(I)$, and $\bar{H}_1=V_{\bar{\B}}^T V_{\bar{\B}}$ which proves $\rank (H_1^+)=\rank(\bar{H}_1)=\bar{k}$ and $\bar{H}_1^+=H_1^{\bar{\B}^+}(\sqrt{\I})$. 
\end{proof}

We can use Algorithm \ref{alg:cert} unchanged with input $f, g$, $\bar{\B}$,  and $\bar{H}_1^+$ to compute  certified Hermite matrices $H_1$ and $H_g$ for $\sqrt{\I}$.  We have the following theorem. 

\begin{theorem}\label{thm:maincert-nonrad}
Let $f=(f_1, \ldots, f_m)\in \Q[x_1, \ldots, x_n]^m$, $\I=\langle f_1, \ldots, f_m\rangle  \subset \Q[x_1,\dots,x_n]$  a zero dimensional ideal, 
	 $g \in \Q[x_1,\ldots,x_n]$. Let
	 $\B$ connected to 1 with $|\B|=k$ and   $|\B^+|=l$  for some $k,l\in \N$. Suppose Algorithm \ref{alg:apprHerm-nonrad} returns $\bar{k}\leq k$,  $\bar{\B}$ connected to 1 with $|\bar{\B}|=\bar{k}$ and 
	 $\bar{H}_1^+$ with rows and columns indexed by the elements of $\bar{\B}^+$. If $\bar{k}\geq \dim\Q[x_1, \ldots, x_n]/\sqrt{\I}$ and Algorithm \ref{alg:cert} with inputs  $f, g$, $\bar{\B}$,  and $\bar{H}_1^+$ returns $H_1$ and $H_g$ then $\bar{\B}$ is a basis for $\Q[x_1, \ldots, x_n]/\sqrt{\I}$ and 
	 $$
	 H_1=H_1^{\bar{\B}}(\sqrt{\I}) \; {\rm and } \;  H_g=H_g^{\bar{\B}}(\sqrt{\I}).
	 $$
\end{theorem}

\begin{proof} Assume Algorithm \ref{alg:apprHerm-nonrad} did not fail, and returns $\bar{k}\leq k$,  $\bar{\B}$ connected to 1 with $|\bar{\B}|=\bar{k}$ and 
	 $\bar{H}_1^+$.   Then by Step 3 of Algorithm \ref{alg:apprHerm-nonrad} we have $\rank(H_1^+)=\rank (\bar{H}_1^+)=\rank(\bar{H}_1)=\bar{k}$ so we can define the $\bar{k}\times \bar{k}$ matrices $M_1, \ldots, M_n$  in Step 2 of Algorithm \ref{alg:cert}. The same argument as in the proof of Theorem \ref{thm:maincert} shows that $M_1, \ldots, M_n$ forms a system of (transpose) multiplication matrices for an ideal $J$ such that $\dim \Q[x_1, \ldots, x_n]/J=\bar{k}$, and $M_1, \ldots, M_n$ are simultaniously diagonalizable and their eigenvalues are the coordinates of the $\bar{k}$ distinct roots of $J$. Since $f_i(M_1, \ldots, M_n)=0$ for all $i=1, \ldots m$ we have that these $\bar{k}$ distinct roots are also  common roots of $\I$. By assumption, $\bar{k}\geq \dim\Q[x_1, \ldots, x_n]/\sqrt{\I}$, so $\I$ has at most $\bar{k}$ distinct roots, so we must have $\bar{k}= \dim\Q[x_1, \ldots, x_n]/\sqrt{\I}$ and $\bar{B}$ forms a basis $\Q[x_1, \ldots, x_n]/\sqrt{\I}$. Thus $J=\sqrt{\I}$. The rest of the proof is the same as the proof of Theorem \ref{thm:maincert} with $\sqrt{\I}$ replacing $\I$. 
\end{proof}


\section{Application: Rational certificate of non-negativity over real varieties}\label{sect:appl1}

As mentioned in the Introduction, the following application was inspired by \cite{CifuPar2017} where the authors give a method to compute a degree $d$ sum of squares (SOS) decomposition (if one exists) for a non-negative polynomial over a real algebraic set using a finite set of sample points from the (complex) algebraic set. In \cite{CifuPar2017}  they also give bounds on the number of sample points needed to decide whether a degree $d$ SOS decomposition exists. Their method uses semidefinite programming.

The application we present below also gives a certificate for the non-negativity of a polynomial  over a real algebraic set, and also uses sample points of the complex algebaric set. However,  it does not give an SOS  decomposition, instead it uses
Hermite matrices reconstructed and certified from the sample points. While the SOS decomposition constructed in \cite{CifuPar2017} is approximate, the construction we give here is exact, using rational numbers. 

Consider the following problem:

\begin{problem}\label{prob:poscert}
Given $f=(f_1, \ldots, f_s) \in \Q[x]^s$  and $g\in \Q[x]$ for $x=(x_1, \ldots, x_n)$.  Decide if $g(z)\geq 0$ for all $z\in V(f)\cap \R^n$  and give a rational certificate.
\end{problem}

We use the following notion of critical points of $g$ in $V(f)\subset \C^n$:

\begin{definition}
    Let $f=(f_1, \ldots, f_s) \in \Q[x]^s$  and $g\in \Q[x]$, and assume that $\I:=\langle f_1, \ldots, f_s\rangle$ is a radical ideal.  $z\in \C^n$ is a {\em  critical point} of $g$ in $V(f)$ if 
    $z\in V(f)$,
     $z$ is non-singular in $V(f)$, i.e. the Jacobian matrix $Jf(z)$ of $f$  has  rank $c:=n-\dim V(f)$,
     and  the Jacobian matrix $Jf(z)$ augmented with the row vector $\nabla g(z)$ has also rank  $c$.
  
\end{definition}

We can solve  Problem \ref{prob:poscert} using the following assumptions on $f=(f_1, \ldots, f_s)$ and $g$:

\begin{assumption}\label{assump} Let $f=(f_1, \ldots, f_s)\in \Q[x]^s$ and $g\in \Q[x]$ for $x=(x_1, \ldots, x_n)$. We assume
\begin{enumerate}
    \item For all $z\in V(f)$, the Jacobian matrix $Jf(z)$ of $f$  has rank $s\le n$, i.e., $f$ is a regular sequence and $V(f)$ is smooth.
    \item $V(f)\cap \R^n$ is bounded,
    \item $g$ has finitely many critical points on $V(f)$.
\end{enumerate}
\end{assumption}

We need the following lemma:

\begin{lemma} \label{lem:Lagr} Let $f\in \Q[x]^s$ and $g\in \Q[x]$ satisfying Assumption \ref{assump}. 
   Then, the set of polynomials in $\Q[x_1, \ldots, x_n, \lambda_1, \ldots, \lambda_s]$
    $$
   L(x,\lambda)=\{f_1, \ldots, f_s\}\cup\left\{\frac{\partial g}{\partial x_i}+ \sum_{j=1}^s\lambda_j\frac{\partial f_j}{\partial x_i}\;:\;i=1, \ldots, n\right\}
    $$
    has finitely many roots in $\C^{n+s}$, and their projections to the $x$ coordinates contain all real points where $g$ attains its extreme values on each connected component of   $\,\left(V(f)\setminus V(g)\right)\cap \R^{n}$.
\end{lemma}

\begin{proof} First we prove that the projection  $ V(L)\subset \C^{n+s}$ onto the $x$-coordinates are the critical points of $g$ in $V(f)$. Let $(x,\lambda)\in V(L)\subset \C^{n+s}$.  Then clearly $x\in V(f)$, and because of our first assumption and that $s\geq c=n-\dim V(f)$ we have that $s=c$ and $x$ is non-singular in $V(f)$. By  $(x,\lambda)\in V(L)$ we also have that  $Jf(x)$ augmented with $\nabla g(x)$ has rank at most $c$, thus $x$ is a critical point.  Conversely, let $x$ be a critical point of $g$ on $V(f)$. Then $f_1, \ldots, f_s$ vanishes at $x$, and since both $Jf(x)$ and the augmented matrix $Jf(x)$ by $\nabla g(x)$ has rank $s$, $\nabla g(x)$ is in the row space of $Jf(x)$, thus there exists $\lambda \in \C^s$ such that $(x,\lambda)\in V(L)$. 

Next we prove that $V(L)\subset \C^{n+s}$ is finite. By assumption, $g$ has finitely many critical points in $V(f)$, thus the projection of $ V(L)$ onto the $x$ coordinates is finite.  Since $Jf(x)$ has full row rank for every $x\in V(f)$, for every critical point $x$ there is a unique $\lambda\in \C^s$ such that $(x,\lambda)\in V(L)$. Thus $V(L)$ is finite. 

To prove the second claim,  assume that  $\,\left(V(f)\setminus V(g)\right)\cap \R^n \neq  \emptyset$ and let $C$ be a  connected component of
the set $(V(f)\setminus  V(g))\cap \R^n$. Since $V(f)\cap \R^n$ is bounded, $C$ is bounded as well.
Since  $C \not\subset V (g)$, there exists  $ x \in C$  with $g(x) \neq 0$. 
Let $\overline{C}$  be the Euclidean
closure of $C$ so that $\overline{C}\subset V \cap \R^n$ is closed and bounded,
and $g$ vanishes identically on $\overline{C}\setminus C$. By the extreme value theorem,
$g$ attains both a minimum and a maximum on $\overline{C}$. Since $g$ is not identically zero on $\overline{C}$,
either the minimum or the maximum value of $g$ on $\overline{C}$ must be nonzero, so~$g$
attains a non-zero extreme value on $C$. By the Lagrange multiplier theorem, using our assumptions, the points in $C$  where $g$ attains its  extreme values are critical points of $g$ in $V(f)$, which proves the claim.
\end{proof}

Now we are ready to prove the main result of the section:

\begin{theorem}\label{thm:J}
Let $f\in \Q[x]^s, g\in \Q[x]$ and $L\subset \Q[x,\lambda]$ as in Lemma \ref{lem:Lagr}. Let $J$ be the ideal generated by $L$ in $\Q[x, \lambda]$, and $\B\subset \Q[x, \lambda]$ be any finite monomial basis for $\Q[x, \lambda]/J$. Then
$$\sigma (H^\B_g) = \sigma(H^\B_{g^2})
\text{ if and only if } g(x)\geq 0 \text{ for all }x\in V(f)\cap \R^n. \footnote{We thank to Bernard Mourrain for pointing out this simple fact}
$$
\end{theorem}

\begin{proof}
 First we prove that $g(x)\ge 0$ for all $x\in V(f)\cap \R^n$ if and only if $g(x,\lambda)\ge 0$ for all $(x,\lambda)\in V(L)\cap \R^{n+s}$. Assume that $g(x)\ge 0$ for all $x\in V(f)\cap \R^n$. Let $(x,\lambda)\in V(L)\cap \R^{n+s}$. Then   $x\in V(f)\cap \R^n$ and since  $g\in \Q[x]$,  $g(x,\lambda)=g(x)\ge 0$. Conversely, assume $g(x,\lambda)\ge 0$ for all $(x,\lambda)\in V(L)\cap \R^{n+s}$. Suppose there exists $x\in V(f)\cap \R^n$  such that $g(x)< 0$. Let $C$ be the bounded connected component of $\,\left(V(f)\setminus V(g)\right)\cap \R^n$ that contains $x$ and $\overline{C}$ its Euclidean closure. Let $x^*\in \overline{C}$ where $g$ attains its minimum on $\overline{C}$. Then $g(x^*)\leq g(x)< 0$, so $x^*\in C$ and $x^*$ is a critical point of $g$ in $V(f)$. Thus there exists  a unique $\lambda\in \R^s$, the solution of a linear system,  such that $(x^*, \lambda)\in V(L)\cap \R^{n+s}$. But then $g(x^*, \lambda)=g(x^*)\ge 0$, a contradiction. 
 
 By Lemma \ref{lem:Lagr} we have $V(L)$ is finite, so $\B$ is finite. Since $L$ has real (rational) coefficients, by Hermite's theorem
 \begin{eqnarray*}
 &&\sigma(H^\B_g)= \#\{(x, \lambda)\in V(L)\cap \R^{n+s}\;:\; g(x)>0\}-\#\{(x, \lambda)\in V(L)\cap \R^{n+s}\;:\; g(x)<0\}\\
 &&\sigma(H^\B_{g^2})= \#\{(x, \lambda)\in V(L)\cap \R^{n+s}\;:\; g(x)>0\}.
 \end{eqnarray*}
 Thus,  $\sigma(H^\B_g)=\sigma(H^\B_{g^2})$ if and only if $\#\{(x, \lambda)\in V(L)\cap \R^{n+s}\;:\; g(x)<0\}=0$, i.e. for all $(x, \lambda)\in V(L)\cap \R^{n+s}$ we have $g(x)\geq 0$.  
 This proves the theorem.
\end{proof}

\begin{corollary}
Let $f \in \Q[x]^s, g\in \Q[x]$ and $L\subset \Q[x,\lambda]$ as in Lemma \ref{lem:Lagr} and  $\B\subset \Q[x, \lambda]$ as in Theorem \ref{thm:J}. Then $H^\B_{g}$ and $H^\B_{g^2}$ have rational entries, and can be computed and certified from approximate roots for the system $L$. Moreover,  $
\sigma (H^\B_g) = \sigma(H^\B_{g^2})
$ can be also certified exactly over the rationals, giving a rational certificate  for  $g(x)\geq 0$ for all $x\in V(f)\cap \R^n$.
\end{corollary}

\begin{proof}
Since $L$ and $g$ has rational coefficients, the matrices $H^\B_{g}$ and $H^\B_{g^2}$ have rational entries. One can find their signature by computing their LU-decomposition for example, again resulting in rational matrices. These give a rational certificate for $
\sigma (H^\B_g) = \sigma(H^\B_{g^2})
$ and by the previous theorem for $g(x)\geq 0$ for all $x\in V(f)\cap \R^n$.
\end{proof}

\begin{remark}
 The size $k$ of the matrices $H^\B_{g}$ and $H^\B_{g^2}$ is equal to the number of critical points of $g$ in $V(f)$. A brute estimate is given by  $|V(L)|\leq d^{n+s}$ with $d=\max \{\deg(g), \deg(f_1), \ldots, \deg(f_s)\}$,  using the Bezout bound for the polynomial system $L$.
\end{remark}

The following proposition is needed in the algorithm below.
\begin{proposition}\label{prop:Frob}
Let $\{z_1, \ldots, z_k\},\{\xi_1,\ldots, \xi_k\} \subset \C^n$ be such that for some $E, M>0$
$$\|\xi_i-z_i\|_2<E \text{ and } \|z_i\|_2\leq M-E\quad i=1, \ldots, k.
$$
Let $\B=\{x^{\alpha_1}, \ldots, x^{\alpha_k}\}\subset \Q[x_1, \ldots, x_n]$ be a set of monomials such that $\deg(x^{\alpha_i})\leq d$ for some $d> 0$. 
Then the Frobenious distance of the Vandermonde matrices is bounded by
$$
\|V_\B(\xi_1, \ldots, \xi_k)-V_\B(z_1, \ldots, z_k)\|_F \leq kndM^{d-1}E.
$$
\end{proposition}

\begin{proof}
As in the proof of Proposition \ref{prop:denom_bound} we can see that for $\alpha\in \N^n$, $|\alpha|\leq d$ we have 
$$
\left|\xi_i^\alpha-z_i^\alpha\right|\leq ndM^{d-1}E.
$$
Taking the Frobenius norm of the difference of the Vandermonde matrices gives  $$\|V_\B(\xi_1, \ldots, \xi_k)-V_\B(z_1, \ldots, z_k)\|_F \leq \sqrt{k^2(ndM^{d-1}E)^2}=kndM^{d-1}E.
$$
\end{proof}

\begin{algorithm} (Certification that $g$ is non-negative over $V(f)\cap \R^n$)
\begin{description}
\item[Input:] $n\in \N$, $f \in \Q[x]^s, g\in \Q[x]$ for $x=(x_1, \ldots, x_n) $ satisfying Assumption \ref{assump}.
\item[Output:] True: $g(z)\geq 0$ for all $z\in V(f)\cap \R^n$\\
 False: $g(z)$ is not $\geq 0$ for all 
        $z\in V(f)\cap \R^n$ \\
        or Fail
     
    \item[~~~~~~1:] Construct $L(x,\lambda)$ as defined in Lemma \ref{lem:Lagr}.
    \item[~~~~~~2:] Compute $\{(z_1,\mu_1) \ldots, (z_k, \mu_k)\}\subset \C^{n+s}$,    finitely many approximate roots  of the real polynomial system $L(x,\lambda)$ as defined in Lemma \ref{lem:Lagr}, together with their  precision $E$ and a bound $M$ such that $\|z_i\|\leq M-E$ for $i=1, \ldots, k$.  Return Fail if $L(x,\lambda)$ has infinitely many roots or $z_i=z_j$ for some $1\leq i<j\leq k$. 
    \item[~~~~~~3:] Compute $\B=\{x^{\alpha_1}, \ldots, x^{\alpha_k}\}$  connected to 1, such that the Vandermonde matrix $V_\B(z_1, \ldots, z_k)$ has smallest singular value is greater than  $kndM^{d-1}E$, where $d$ is the maximal degree of the monomials in $\B$ (see Proposition \ref{prop:Frob}). Let  $\B^+:=\bigcup_ix_i\B\cup \bigcup_j\lambda_j\B$. 
    \item[~~~~~~4:] $H_1^+\leftarrow \text{\sc Hermite Matrix Computation}\left(\B, \B^+, E, M, \{(z_1,\mu_1), \ldots, (z_k, \mu_k)\}\right)$ (see Algorithm \ref{alg:apprHerm})
    \item[~~~~~~5:] For $J:=\langle L(x,\lambda)\rangle$ compute  $H_g(J)$ and $H_{g^2}(J)$ by calling  {\sc Hermite Matrix Certification} with input $L(x,\lambda), g, \B, \B^+, H_1^+ $ and $L(x,\lambda), g^2, \B, \B^+, H_1^+ $ respectively, which algorithm can also return Fail (see Algorithm \ref{alg:cert})
    \item[~~~~~~6:] Calculate the signatures $\sigma(H_g(J)$ and $\sigma(H_{g^2}(J))$ (see Remark \ref{rmk:signature_compt})
    \item[~~~~~~~7:] if $\sigma(H_g(J))$=$\sigma(H_{g^2}(J))$, then return True\\
    else return False

\end{description}
\end{algorithm}

\section{Application: Real Root Certification with Hermite Matrices}\label{sect:appl2}

Let $\ff=(f_1, \ldots, f_m)$  
$\in \Q [x_1,\dots,x_n]^m$  such that  $\I=\langle f_1, \ldots, f_m\rangle $ is a zero dimensional radical ideal. The certification problem is the following: We are given   $z^* \in \Q^n$ and $\varepsilon\in\Q_+$ , we would like to know if there is any exact root of $\ff$ within the $\varepsilon$ ball of  $z^*$ in $\RR^n$.
In this section, we present how one can use the signature of Hermite matrices to answer this certification problem as an application of Hermite matrices.

Now we introduce the certification theorem as a corollary of the Multivariate Hermite Theorem \ref{thm:mult_herm}.

\begin{corollary}\label{thm:herm_cert}
	Let $\ff=(f_1, \ldots, f_m)$, $f_i \in \Q[x_1,\dots,x_n]$ for all $i=1,\dots m$, and $\I=\langle f_1, \ldots, f_m \rangle$ is a zero dimensional radical ideal. Given  $z^* \in \Q^n$ and  $\varepsilon \in \Q_+$,	define $g(x):=\|x-z^* \|_2^2-\varepsilon^2\in \Q[x_1, \ldots, x_n]$. Then
	$$ \sigma(H_1(\I)) = \sigma(H_g(\I))$$
	if and only if 
	there is no real root within the closed ball in $\R^n$  of radius $\varepsilon$ around $z^*$. 
\end{corollary}

\begin{proof}

By Theorem \ref{thm:mult_herm}, $\sigma(H_g(\I)) = \#\{x \in V_\RR(\I)~|~ g(x)>0\}-\# \{ x \in V_\RR(\I)~|~ g(x)<0\} $. Since 1 is greater than 0, the signature of $H_1(\I)$ gives the number of all real roots of $\I$. Thus, it can be written as 
\begin{equation}\label{eqn:H1I}
\sigma(H_1(\I))=\#\{x \in V_\RR(\I)~|~ g(x)>0\}+\# \{ x \in V_\RR(\I)~|~ g(x)<0\}+\#\{x \in V_\RR(\I)~|~ g(x)=0\}  
\end{equation}
 for any $g(x)\in\R[x_1, \ldots, x_n]$.
Now, let  $g(x)=\|x-z^* \|_2^2-\varepsilon^2\in \R[x_1, \ldots, x_n]$.\\
Assume that $\sigma(H_1(\I)) = \sigma(H_g(\I))$. Using the definitions above
and cancelling the terms $ \#\{x \in V_\RR(\I)~|~ g(x)>0\}$ from both sides, we yield $ 2\# \{ x \in V_\RR(\I)~|~ g(x)<0\} = -\#\{x \in V_\RR(\I)~|~ g(x)=0\}.$ The only solution to this equation is 
$$ \# \{ x \in V_\RR(\I)~|~ g(x)<0\} =0 \text{  and  } \#\{x \in V_\RR(\I)~|~ g(x)=0\}=0,$$
since these terms are nonnegative integers by definition. This implies that there is no real solution to $\I$, when $g(x)<0$ and $g(x)=0$. By the definition of $g(x)=\|x-z^* \|_2^2-\varepsilon^2$, there is no $x\in \RR^n$ such that $\|x-z^* \|^2_2 \leq \varepsilon^2$. Thus we conclude that there is no real root within the closed ball of radius $\varepsilon$.\\
Assume that there is no real root to $\I$ within the closed ball in $\R^n$  of radius $\varepsilon$ around $z^*$. It implies that any $x \in V_\R(\I)$ satisfies the inequality $\|x-z^* \|_2^2 >\varepsilon^2$. Then 
$g(x)=\|x-z^* \|_2^2-\varepsilon^2$ is always positive, and $\#\{x \in V_\RR(\I)~|~ g(x)\leq 0\}=0$. By Theorem \ref{thm:mult_herm} and (\ref{eqn:H1I}), we have
$$
\sigma(H_1(\I)) = \sigma(H_g(\I))=\#\{x \in V_\RR(\I)~|~ g(x)> 0\}.
$$
\end{proof}
\noindent




\begin{algorithm}[Real Root Certification]\label{alg:root_cert}
$\;$
\begin{description}
	\item[Input:]
	 $\ff=(f_1, \ldots, f_m) \in \Q[x_1,\dots,x_n]^m$;
	 $z^* \in \Q[i]^n$;
	 $\varepsilon^2 \in \Q_+$;\\
	 $\B=\{x^{\alpha_1}, \ldots, x^{\alpha_k}\}$ connected to 1  with $|\B^+|=l$  for some $k,l\in \N$\\
	 $E, M\in\R_+$ and $z_1, \ldots, z_k\in \C^n$ such that $\|z_i\|_\infty \leq M-E$ for $i=1, \ldots, k$ and the accuracy of $z_i$ is at least $E$.
	 
	 \item[Output:] True: $\exists z\in V_\R(\I)$ such that $z$ is in the closed ball   of radius $\varepsilon$ around $z^*$\\
	 False: No  real root of $\I$  within the closed ball   of radius $\varepsilon$ around $z^*$\\
	 or Fail.

	\item[~~~~~~1:] Define $g(x):=\|x-z^* \|_2^2-\varepsilon^2\in \Q[x_1, \ldots, x_n]$
	
	\item[~~~~~~2:]  $H_1^+:=H_1^{\B^+}(z_1, \ldots,z_k)
	\leftarrow \text{\sc Hermite Matrix Computation}\left(\B, \B^+, E, M, \{z_1, \ldots, z_k\}\right)$\\ (see Algorithm \ref{alg:apprHerm})

	
	\item[~~~~~~3:] For $I:=\langle f_1, \ldots, f_m\rangle$ call
	$\text{\sc Hermite Matrix Certification}
	(f,g(x), B, H_1^+)$ to obtain certified $H_1(I)$ and $H_g(I)$, that algorithm can also return Fail (see Algorithm \ref{alg:cert})
	
	\item[~~~~~~4:] Compute $\sigma(H_1(\I))$ and $\sigma(H_g(\I))$. See Remark \ref{rmk:signature_compt} for computational details.
	\item[~~~~~~5:] {\bf If} $\sigma(H_1(\I)) = \sigma(H_g(\I))$ 
	{\bf then} return False\\ {\bf else} return True.

	\end{description}
\end{algorithm}

%


\begin{thebibliography}{100}


\bibitem
{Apostol1969}
Apostol, T. M. : Calculus, volume II: multi-variable calculus and linear algebra, with applications to differential equations and probability. John Wiley \& Sons (1969)

\bibitem{AS2020}
Ayyildiz Akoglu, T., and Szanto, A. : Certified Hermite Matrices from Approximate Roots-Univariate Case. MACIS 2019, Lecture Notes in Computer Science book series, vol.  \textbf{11989}, pp. 3--9 (2020).



\bibitem
{AHS2018}
Ayyildiz Akoglu, T., Hauenstein, J.D., Szanto, A. : Certifying solutions to overdetermined and singular polynomial systems over $\Q$. Journal of Symbolic Computation, \textbf{84}, 147--171 (2018)



\bibitem
{AA2016}
Ayyildiz Akoglu, T. : Certifying solutions to polynomial systems over $\Q$. PhD thesis, North Carolina State University (2016)

\bibitem{BFS2004}
Bardet, M., Faugère, JC. and Salvy, B. :
\newblock On the complexity of {G}r\"obner basis computation of semi-regular
  overdetermined algebraic equations.
\newblock In {\em Proceedings of the {I}nternational {C}onference on
  {P}olynomial {S}ystem {S}olving}, ICPSS'04, pages 71--75, 2004.

\bibitem{BFSY2005}
Bardet, M., Faugère, JC., Salvy, B. and Yang, BY. :
\newblock
Asymptotic Behavior of the Degree of Regularity of Semi-Regular Polynomial Systems.
\newblock in MEGA 2005 , Eighth International Symposium on Effective Methods in Algebraic Geometry,  15 p., 2005.

\bibitem
{BPR2006}
Basu, S., Pollack, R., Roy, M.-F. : Algorithms in real algebraic geometry. Springer-Verlag, Berlin Heidelberg 
(2006)

\bibitem{Batesetal2014}
Bates, D.J., Decker, W., Hauenstein, J.D., Peterson, C., Pfister, G., Schreyer, F.O., Sommese, A.J. and Wampler, C.W., :  Comparison of probabilistic algorithms for analyzing the components of an affine algebraic variety. Applied Mathematics and Computation, 231, pp.619-633.(2014)

\bibitem
{Blumetal1998}
Blum, L., Cucker, F., Shub, M., Smale, S.:
Complexity and real computation. 
With a foreword by Richard M. Karp. Springer-Verlag, New York, (1998). 

\bibitem
{CifuPar2017}
Cifuentes, D. and Parrilo, P.  Sampling algebraic varieties for sum of squares programs. SIAM J. Optimization. 27(4):2381-2404. (2017).

\bibitem
{CLO2006}
Cox, D. A., Little, J., O'shea, D. : Using algebraic geometry (Vol. 185). Springer Science \& Business Media (2006)


\bibitem
{HSW2013}
Bates, D.J., Hauenstein, J.D., Sommese, A.J., Wampler, C.W., : Numerically solving polynomial systems with Bertini (Vol. 25). SIAM, (2013)

 
\bibitem
{Hauenstein-Sottile}
Hauenstein, J.D., Sottile F..
\newblock Algorithm 921: alpha{C}ertified: certifying solutions to polynomial
  systems.
\newblock {\em ACM Trans. Math. Software}, 38(4):28, 2012.


\bibitem
{Hermite1850}
Hermite, C. : Sur le nombre des racines d'une {\'e}quation alg{\'e}brique comprise entre des
limites donn{\'e}es, J. Reine Angew. Math. 52 (1850), 39-51; also in Oeuvres completes,
Vol. 1, pp. 397-414.

\bibitem
{Hermite1853}
Hermite, C. : Remarques sur le théorème de Sturm. CR Acad. Sci. Paris, 36(52-54), 171 (1853)

\bibitem
{Hermite1856}
Hermite, C. : Extrait d'une lettre de Mr. Ch. Hermite de Paris {\`a} Mr. Borchardt de Berlin sur le nombre des racines d'une {\'e}quation alg{\'e}brique comprises entre des limites donn{\'e}es. Journal f{\"u}r die reine und angewandte Mathematik \textbf{52} 39--51 (1856)

\bibitem
{JanRonSza2007}
Janovitz-Freireich, I., R\'onyai, L.,  Sz\'ant\'o, A. : Approximate Radical for Clusters: A Global Approach Using Gaussian Elimination or SVD, Math.comput.sci. (2007) 1: 393. 

\bibitem
{Macdonald79}
Macdonald, I. G. : Symmetric functions and Hall polynomials. Oxford Mathematical Monographs. Oxford: The Clarendon Press, Oxford University Press. 1979.

\bibitem
{Mourrain1999}
Mourrain, B. : A new criterion for normal form algorithms. In M. Fossorier, H. Imai, Shu Lin, and A. Poli, editors, Proc. AAECC, volume 1719 of LNCS, pp. 430--443. Springer, Berlin, 1999.


\bibitem
{PRS1993}
Pedersen, P., Roy, M. F.,  Szpirglas, A. : Counting real zeros in the multivariate case. In Computational algebraic geometry (pp. 203-224). Birkhäuser, Boston, MA (1993)



\bibitem
{rosen1993}
Rosen, K. H. : Elementary number theory and its applications. 6th Edition. Reading, Massatchusets (1993)

\bibitem
{schrijver1998}
Schrijver, A. : Theory of linear and integer programming. JohnWiley \& Sons, New York (1998)


\bibitem{Sti1890} 
Stickelberger, L. : Ueber eine Verallgemeinerung der Kreistheilung. Math. Ann., 37, pp. 312–
367  (1890)














\end{thebibliography}
\end{document}